\documentclass[a4paper,twocolumn]{article}
\usepackage{amsmath}
\usepackage{amssymb}
\usepackage{amsthm}
\usepackage{amsfonts}
\usepackage{authblk}
\usepackage[english]{babel}
\usepackage{float}
\usepackage{textcomp}
\makeatletter
\usepackage{graphicx}
\usepackage{wrapfig}
\usepackage[misc,geometry]{ifsym}
\usepackage[table]{xcolor}
\usepackage{mathtools}
\usepackage{natbib}
\def\m{\phantom{-}}
\newcommand{\eg}{\emph{e.g.,\ }}
\newcommand{\ie}{\emph{i.e.,\ }}
\newtheorem{lemma}{Lemma}
\newtheorem{proposition}{Proposition}
\definecolor{kthgrey}{RGB}{227,229,227}
\DeclarePairedDelimiter\norm{\lVert}{\rVert}%
\let\oldnorm\norm
\def\norm{\@ifstar{\oldnorm}{\oldnorm*}}
\providecommand{\keywords}[1]{\textbf{Keywords:} #1}
\begin{document}
\title{On Dynamically Generating Relevant Elementary Flux Modes in a Metabolic Network using Optimization}
\author[1]{Hildur {\AE}sa Oddsd\'{o}ttir \thanks{Corresponding author: haodd@kth.se}}
\author[2]{Erika Hagrot}
\author[2]{V\'{e}ronique Chotteau }
\author[1]{Anders Forsgren}
\affil[1] {Department of Mathematics, Optimization and Systems Theory, 
KTH Royal Institute of Technology, 
SE-100 44 Stockholm, Sweden}
\affil[2]{Division of Industrial Biotechnology/Bioprocess Design, KTH Royal Institute of Technology, 
Albanova Center,
SE-106 91 Stockholm, Sweden}
\date{Final publication available at Springer via http://dx.doi.org/10.1007/s00285-014-0844-1}
\maketitle
\begin{abstract}
Elementary flux modes (EFMs) are pathways through a metabolic reaction network that connect external substrates to products. Using EFMs, a metabolic network can be transformed into its macroscopic counterpart, in which the internal metabolites have been eliminated and only external metabolites remain. In EFMs-based metabolic flux analysis (MFA) experimentally determined external fluxes are used to estimate the flux of each EFM. It is in general prohibitive to enumerate all EFMs for complex networks, since the number of EFMs increases rapidly with network complexity. In this work we present an optimization-based method that dynamically generates a subset of EFMs and solves the EFMs-based MFA problem simultaneously. The obtained subset contains EFMs that contribute to the optimal solution of the EFMs-based MFA problem. The usefulness of our method was examined in a case-study using data from a Chinese hamster ovary cell culture and two networks of varied complexity. It was demonstrated that the EFMs-based MFA problem could be solved at a low computational cost, even for the more complex network. Additionally, only a fraction of the total number of EFMs was needed to compute the optimal solution.
\let\thefootnote\relax\footnote{\emph{Abbreviations:} MFA, metabolic flux analysis; EFMs, elementary flux modes; CHO, Chinese hamster ovary; Lac, lactate; Glc, glucose}
\end{abstract}

\section{Introduction}
Modeling the metabolism of a cell can give deeper understanding of how the cell responds to different conditions. The cells consume substrates from their environment. These are transformed into products, cell building blocks, and energy through biochemical reactions. The products are excreted, while the cell building blocks and energy are used for cell division.
A set of biochemical reactions can be used to represent cell metabolism by a metabolic reaction network; essentially a set of metabolites and the reactions interconnecting them. 
The stoichiometric coefficients are contained in the stoichiometric matrix ($A$), where rows correspond to metabolites and columns to reactions. The overall change in the concentration of each metabolite ($C$) can then be expressed as the product of the stoichiometric matrix and the flux vector ($v$) consisting of the fluxes on each reaction ($v_j$), 
\begin{equation*}
\frac{dC}{dt}=Av.
\end{equation*}
Several metabolites, called external metabolites ($C_x$), are present outside the cell and measured in experiments. The remaining metabolites, called internal metabolites ($C_i$), are present inside the cell and typically not measured.
The rows of $A$ can then be split into two sets, $A_i$ and $A_x$, depending on if a metabolite $C$ is an internal metabolite ($C_i$), or external metabolite ($C_x$), respectively. 
The assumption of pseudo-steady state is that
\begin{equation} \label{eq:pss}
\frac{dC_i}{dt}=A_iv=0.
\end{equation}
Metabolic flux analysis (MFA) and metabolic pathway analysis are two methodologies based on \eqref{eq:pss} \citep{Llaneras2008}. Flux vectors ($v$) that fulfill \eqref{eq:pss} are said to be in the flux space.

In MFA it is desirable to find the unique flux vector that satisfies $A_i v=0$, $A_x v=q_{ext}$, and $v \geq 0$, where $q_{ext}$ contains measurements of the external fluxes and $v \geq 0$ indicates that $v_j \geq 0, \, \forall j \in J$, where $J$ is the set of all reactions.
This problem as stated does in general not give a unique solution, hence assumptions are often made in order to find the flux \citep{Bonarius1997,Klamt2002}.

Metabolic pathway analysis has its basis in convex analysis \citep{Clarke1980}. The stoichiometry of the network is examined by defining a finite set of pathways that gives a basis for the flux space, \ie any feasible flux can be expressed as a non-negative linear combination of these pathways \citep{Schilling1999,Papin2003}. 
Various concepts that generate the flux space have been suggested, \eg extreme currents, elementary flux modes (EFMs), generating flux modes, and extreme pathways. For a comparison see \citep{Klamt2003,Papin2004,Planes2008,Rezola2011}. 

This paper will consider EFMs; a single EFM is denoted by the vector $e$ and the matrix $E$ contains the EFMs as columns. 
EFMs form a finite set of metabolic pathways that satisfy the pseudo-steady state assumption ($A_{i}e=0$) and use each reaction in the direction given, $e_j\geq 0$, $j\in J_{\text{irrev}}$, where $J_{\text{irrev}}$ is the set of irreversible reactions. Furthermore, each flux can be constructed from the set of EFMs, see Section \ref{sec:EFMs} for a more detailed description.
EFMs can be considered as pathways through the network connecting external substrates to products. 
Reversible fluxes are handled in a way that facilitates interpretation; simple inspection of EFMs 
can help answer questions such as which substrates and products are linked or which reactions are required for the production of a certain metabolite \citep{Llaneras2010}. 
EFMs can be enumerated through computer programs, \eg Metatool \citep{VonKamp2006}. 
Examples of applications using EFMs are: analyzing the metabolic pathways in a cell \citep{Schilling2000}; singular value decomposition analysis \citep{Price2003a}; and the development of dynamic macroscopic models \citep{Provost2006a}.

EFMs-based MFA combines the fields of metabolic flux analysis and metabolic pathway analysis. It provides a way to quantify the flux over each EFM. A macroscopic network ($A_xE$) is defined from the EFMs, in which each macroscopic reaction connects external substrates and products. The macroscopic network has a corresponding flux vector $w$, given by $v=Ew$, that consists of the fluxes over each macroscopic reaction. The EFMs-based MFA problem is to fit the macroscopic metabolic network to the cell specific external flux measurements ($q_{ext}$) by adjusting the macroscopic fluxes ($w$). The specific problem is given by, 
\begin{equation}
\begin{aligned}
\underset{w}{\text{minimize}} \quad & \frac{1}{2} \|q_{ext}-A_xE w \|_2^2,\\
\text{subject to }\quad & w \geq 0,
\end{aligned} \label{eq:EFMsbMFA1}
\end{equation}
where the flux on each EFM, $w$, is determined \citep[Chapter 5.2]{Provost2006}. 

As the network size increases the number of EFMs, or any other minimal description like extreme pathways, explodes and enumeration becomes prohibitive \citep{Klamt2002a}. 
In EFMs-based MFA the whole set of EFMs is enumerated beforehand, making its application limited to simplified networks.
To enable EFM analysis of complex metabolic networks, methods that find only subset of the EFMs have been developed.
The authors of \citep{DeFigueiredo2009} presented a method that finds EFMs through a mixed integer linear program. A predetermined number of EFMs that use the fewest reactions are identified, meaning that the EFMs using a high number of reactions may seldom be included in this set.
Another method proposed by \citep{Kaleta2009}, and later improved on by \citep{Tabe-Bordbar2013}, uses a sequence of linear optimization problems to find EFMs connected to specific external metabolites. Identifying EFMs through a series of linear programming (LP) problems was also proposed by \citep{Jungers2011}, where EFMs are identified in a random manner until the predetermined minimal number of EFMs are found. 
This method could be used to solve the EFMs-based MFA problem where multiple measurements are available; for that three steps need to be taken. First, the flux vector $v$ is identified through least-squares data fitting, then the method could be used to decompose $v$ into the EFMs needed, and finally, the macroscopic flux vector $w$ can be calculated. 

In this work we considered solving the EFMs-based MFA problem by a more integrated approach that does not require the enumeration of EFMs or estimation of the flux vector $v$ beforehand. Instead, we imagined that only the EFMs needed to solve the problem could be extracted by using all external measurements directly; this should enable the problem to be solved even when the network is so complex that total enumeration of EFMs is prohibitive. 
To achieve this we suggest a method that identifies EFMs in a dynamic fashion. The method is based on an optimization technique known as column generation, which enables EFMs to be identified in conjunction with solving the EFMs-based MFA problem. Column generation is an established technique within optimization \citep{Lubbecke2005} where two levels of optimization problems are solved iteratively, a master problem and a subproblem. In this work the master problem is a quadratic optimization problem on the same form as the EFMs-based MFA problem, however, only a subset of the EFMs is used. The subproblem is an LP problem that identifies EFMs that further improve the data fitting in the master problem. Hence only the most relevant EFMs with respect to external measurements should be found. 

The outline of the paper is as follows. In Section \ref{sec:th} we present the theory necessary to derive the proposed method. In Section \ref{sec:Method} we present the proposed method, show that it indeed finds EFMs of the network and explain how reversible reactions are considered. The usefulness of our approach is demonstrated in Section \ref{sec:CS} by using data obtained in Chinese hamster ovary (CHO) cell culture. Finally, Section \ref{sec:conc} gives our conclusions.
\section{Theory}\label{sec:th}
\subsection{Metabolic Network Cones}\label{sec:Mcones} 
To simplify the exposition, the metabolic network is initially assumed to only have irreversible reactions, reversible reactions are studied in Section \ref{sc:irrev}. The pseudo-steady state assumption from \eqref{eq:pss} together with the irreversibility of each reaction gives the flux space. The flux space, or cone, is the set of vectors $v$ that satisfy, 
\begin{equation}\label{eq:Cone}
\hat{A}v=\begin{bmatrix}
\m A_{i}\\
-A_{i}\\
-I
\end{bmatrix}v \leq \begin{bmatrix} 0\\0\\0
\end{bmatrix}, 
\end{equation}
where $I$ is the identity matrix. 

The cone given by \eqref{eq:Cone} has full column rank and hence has the $0$ vector as an extreme point. If there are any other points in the cone those are a part of a ray \citep[Part I.4 Definition 4.2]{Nemhauser1999}. 
An extreme ray is a ray in the cone that cannot be written as a non-negative linear combination of two other rays in the cone. In this sense the extreme rays define the cone, since any ray in the cone can be written as a non-negative linear combination of the extreme rays. Hence, if all the extreme rays are gathered as columns in a matrix $R$, any ray $r$ can be written as $r=R\gamma$, where $\gamma \geq0$ \citep[Part I.4 Theorem 4.8]{Nemhauser1999}. 
\subsection{Elementary Flux Modes} \label{sec:EFMs}
An elementary flux mode (EFM) is defined by \citet{Klamt2002a} as a flux vector that satisfies the pseudo-steady state assumption, is nondecomposable, and includes only a minimal set of reactions. EFMs are closely related to the concept of extreme rays. All flux vectors ($v$) in the network can be written as a non-negative linear combination of the EFMs, 
\begin{equation}
v= \sum_{l=1}^L \gamma_l e_l =E\gamma, \qquad \gamma \geq 0,
\label{eq:combin}
\end{equation} 
where $e_l$ is an EFM and $E$ is a matrix with the EFMs as columns.
A formal definition of an EFM is given by \citet{Schuster1994}.
For a metabolic network with only irreversible reactions there is a simple interpretation of the EFM's definition in mathematical terms.
\begin{lemma}[\cite{Gagneur2004}, Lemma 1] \label{lem:EREFM}
When a metabolic network only has irreversible reactions the EFMs and the extreme rays of the cone \eqref{eq:Cone} are equal.
\end{lemma}
Hence a method that finds the extreme rays of the cone given by \eqref{eq:Cone} also finds the EFMs. 
When there are reversible reactions it has been shown that finding all the EFMs is equivalent to finding all the extreme rays of a cone in an extended space \citep{Gagneur2004,Urbanczik2005}. Therefore, it can be assumed without loss of generality that the metabolic network has only irreversible reactions. 
\subsection{The EFMs-based MFA problem}
In the EFMs-based MFA problem, macroscopic fluxes ($w$) are sought. The macroscopic fluxes ($w$) then correspond to the weights $\gamma$ in \eqref{eq:combin} and each flux vector ($v$) can be written as a non-negative linear combination of the EFMs,
\begin{equation}
v= \overset{L}{\underset{l=1}{\sum}} w_l e_l =Ew, \; w \geq 0. 
\label{eq:vdec}
\end{equation}
Using the decomposition given by \eqref{eq:vdec}, the external measurements ($q_{ext}$) can be linked to the macroscopic fluxes $w$. 
In $A_xv \approx q_{ext}$, $v$ can be substituted with the representation from \eqref{eq:vdec} giving,
\begin{equation*}
A_xv=A_xEw \approx q_{ext}.
\end{equation*}
Similarly, the internal stoichiometry ($A_i$) cancels out, 
\begin{equation*}
A_{i}v=A_{i}Ew =0, \quad \text{since } A_{i}E=0, \text{ by def. of E.}
\end{equation*}

In EFMs-based MFA, $w$ is obtained through least-squares regression as shown by \eqref{eq:EFMsbMFA1}. When external measurements of multiple repetitions of an experiment are taken into account and arranged in a vector $Q$, the least-squares regression to approximate $w$ is to 
\begin{equation}\label{eq:simpTot}
\begin{aligned}
\underset{w}{\text{minimize}} \quad & \frac{1}{2} \|Q-\mathcal{I}A_xE w \|_2^2,\\
\text{subject to } \quad 
& w \geq 0.
\end{aligned}
\end{equation}
More precisely, if $q_{ext,k}$ are results from one repetition, $k$, then $Q^T= [q_{ext,1}^T, \ldots q_{ext,d}^T]$, where $d$ denotes the number of repetitions.
$\mathcal{I}$ is a stacked identity matrix consisting of $d$ identity matrices of size $M_{ext}$ (number of external metabolites) or $\mathcal{I}=[I_{M_{ext}}, \ldots, I_{M_{ext}}]^T$, where $I_{M_{ext}}$ is repeated $d$ times.
\section{The Proposed Method}\label{sec:Method}
Problem \eqref{eq:simpTot} assumes that the whole set of EFMs $E$ is known. We suggest a method to solve the problem when $E$ is not known beforehand. The columns of $E$ needed to minimize the problem are generated dynamically using a column generation technique. 
Two problems, the master problem (MP) and the subproblem (SP), are solved iteratively. 
One iteration involves obtaining the optimal solution vector $w_B^*$ from the master problem. Using that solution the subproblem can be stated and solved. The subproblem either finds a new EFM that improves the objective of the master problem or indicates that the solution of the master problem cannot be improved through the addition of more EFMs. If an improving EFM is found, that EFM is added to $E_B$ and the master problem is solved again. The master and subproblem will be introduced in the following discussion, while further derivation of the optimization problems can be found in Section \ref{sc:MPSP}.

The master problem is a linear least-squares data fitting problem, where the external measurements ($Q$) are fitted to the known EFMs ($E_B$) by adjusting the macroscopic fluxes ($w_B$) from the optimization problem
\begin{equation}\label{eq:MP}\tag{MP}
\begin{aligned}
\underset{w_B}{\text{minimize}} \quad & \frac{1}{2} \|Q-\mathcal{I}A_xE_B w_B\|_2^2,\\
\text{subject to } \quad & w_B \geq 0,
\end{aligned}
\end{equation}
where the subscript $B$ denotes the subset of EFMs included.

The subproblem uses the optimal solution $w_B^*$ from the master problem to find, if possible, a new pathway $e$ that can further improve the solution of the master problem, \ie decrease the norm of \eqref{eq:MP}.
\begin{equation}\label{eq:SP} \tag{SP}
\begin{aligned}
\underset{e}{\text{minimize}} \quad & c^Te,\\
\text{subject to } \quad 
& A_{i}e=0, \;
\mathbf{1}^Te \leq 1, \;
e \geq 0, 
\end{aligned}
\end{equation}
where $c= A_x^T\mathcal{I}^T (\mathcal{I}A_xE_B w_B^* -Q)$ and $\mathbf{1}$ is a vector consisting only of ones. If the optimal value is negative then the optimal EFM, $e^*$, found should be included in the master problem. If the optimal value is zero, then $E_B$ includes all EFMs that are needed to minimize \eqref{eq:simpTot} and the problem has been solved.

To simplify the exposition, the two optimization problems are described for an irreversible metabolic network, \ie where no reaction is reversible. Reversible reactions can be handled in the same framework, as outlined in Section \ref{sc:irrev}.

\subsection{Deriving the Master and Subproblem}\label{sc:MPSP}
We first describe how the master problem is formed. 
Assume that \eqref{eq:simpTot} has been solved for a known subset of the columns of $E$ where the $w$ components of the unknown columns of $E$ are fixed to $0$. 
This means that if $E$ and $w$ are split into the known subset, index $B$, and the unknown subset, index $N$, $E=[E_B \, E_N]$, $w^T=[w_B^T \, w_{N}^T]$, then $w_N=0$ is fixed. 
The resulting problem is the master problem \eqref{eq:MP}.

The derivation of the subproblem is done through considering the optimality conditions of \eqref{eq:simpTot} and \eqref{eq:MP}. When \eqref{eq:simpTot} has been solved to optimum the solution, $w^*$, satisfies the following optimality conditions (see \eg \citet[Chapter 14.5]{Griva2009}), 
\begin{equation}\label{eq:Simpredoptcond}
\begin{aligned}
& w^* \geq 0, \quad
\lambda \geq 0, \quad
\lambda^Tw^*=0, \\
& E^T A_x^T \mathcal{I}^T(\mathcal{I}A_xEw^*-Q) =\lambda.
\end{aligned}
\end{equation}
The optimality conditions of \eqref{eq:simpTot}, given by \eqref{eq:Simpredoptcond}, are both necessary and sufficient for a globally optimal solution, since it is a quadratic programming problem. This implies that any $w$ satisfying the optimality conditions \eqref{eq:Simpredoptcond} is globally optimal to \eqref{eq:simpTot}.
The master problem \eqref{eq:MP} satisfies the same optimality conditions as given by \eqref{eq:Simpredoptcond} where $w=w_B$ with two additional constraints, namely \eqref{eq:wN0} and \eqref{eq:OptFSubp},
\begin{subequations}\label{eq:OptCSubpall}
\begin{align}
& w^*_B \geq 0, \quad \lambda_B \geq 0, \quad
\lambda_B^Tw^*_B=0,\\
& w^*_N=0, \label{eq:wN0} \\
& E_B^T A_x^T \mathcal{I}^T(\mathcal{I}A_xE_Bw^*_B-Q) =\lambda_B.\\
& E_N^T A_x^T\mathcal{I}^T (\mathcal{I}A_xE_B w^*_B -Q)=\lambda_N. \label{eq:OptFSubp}
\end{align}
\end{subequations}
When \eqref{eq:Simpredoptcond} and \eqref{eq:OptCSubpall} are compared, it can be noted that if $\lambda_N \geq 0$ in \eqref{eq:OptFSubp} then the conditions of \eqref{eq:Simpredoptcond}, are also fulfilled, \ie the optimal solution of \eqref{eq:MP} is optimal to \eqref{eq:simpTot}.
However, if there is an element in $\lambda_N$ that is negative then the solution of the master problem \eqref{eq:MP} is not optimal to the full problem \eqref{eq:simpTot}.
Instead of enumerating all $E_N$, an optimization problem that finds a negative $\lambda_N$ could be solved.
For that reason \eqref{eq:OptFSubp} defines the objective function of the subproblem, \ie the function to be minimized in order to find a new EFM. Using a linear objective the subproblem can be stated in a somewhat abstract form, 
\begin{equation}\label{eq:SPabs}
\begin{aligned}
\underset{e}{\text{minimize}} \quad & c^Te,\\
\text{subject to } \quad & 
e \in E, 
\end{aligned}
\end{equation}
where $c=A_x^T\mathcal{I}^T (\mathcal{I}A_xE_B w^*_B -Q)$.
Here the constraints are only given in the form that $e$ should be included in the set of EFMs. 
For any $e^*$ optimal to \eqref{eq:SPabs} with a negative objective value it holds that $e^*\in E_N$, since all $e \in E_B$ satisfy \eqref{eq:Simpredoptcond}. 
Furthermore, if $c^Te^* = 0$ then there is no $\lambda_N<0$ in \eqref{eq:OptFSubp} and the solution of \eqref{eq:MP} is a globally optimal solution to the full problem \eqref{eq:simpTot}.

In order to solve \eqref{eq:SPabs} the constraints need to be explicitly stated. The EFMs should satisfy the pseudo-steady state assumption and that all reactions are irreversible,
\begin{equation}\label{eq:ConstSubp}
\begin{aligned}
A_ie=0, \quad &
e \geq 0.
\end{aligned}
\end{equation}
The constraints given by \eqref{eq:ConstSubp} define a cone. 
Hence any vector $e$ that satisfies the constraints \eqref{eq:ConstSubp} and gives a negative objective can decrease indefinitely, \ie the LP with only the constraints given by \eqref{eq:ConstSubp} will be unbounded below and cannot be solved. 
Hence a constraint that bounds the problem is added. Any single constraint that is guaranteed to be active at optimality and bounds the problem can be used. In this formulation bounding the 1-norm of $e$, or 
\begin{equation}\label{eq:ConstSubpBound}
\mathbf{1}^T e \leq 1,
\end{equation}
was chosen. Giving the objective function a lower bound will however deliver similar results.
The subproblem is thus, to minimize the objective function from \eqref{eq:SPabs} subject to the constraints given by \eqref{eq:ConstSubp} and \eqref{eq:ConstSubpBound}. 

\subsection{The Subproblem and EFMs}
In this section we will show that an $e^*$ optimal to the subproblem gives an EFM. It should be noted that a similar optimization problem has been used previously to find EFMs \citep{Acuna2009}. We will however give a full description for the sake of completeness.
The flux cone and a flux polyhedron will be used. 
The flux cone, now stated again was previously given by \eqref{eq:Cone},
\begin{equation}\label{eq:PolyP} \tag{FC}
\begin{aligned} \{e: &&
A_{i}e \leq 0, &&
-A_{i}e \leq 0, &&
-e \leq 0 \}. &&
\end{aligned}
\end{equation}
We define \eqref{eq:PolyP1} as the polyhedron arising when the cone \eqref{eq:PolyP} has been limited with respect to the 1-norm of $e$.
\begin{equation}\label{eq:PolyP1}\tag{FP}
\begin{aligned} \{e: &&
A_{i}e \leq 0, &&
-A_{i}e \leq 0, &&
-e \leq 0, &&
\mathbf{1}^Te \leq 1\}.
\end{aligned}
\end{equation}
Note that \eqref{eq:PolyP1} defines the feasible region of the subproblem \eqref{eq:SP}.
A polyhedron $P=\{v \in \mathcal{R}^n: Dv\leq b\}$ is guaranteed to have an optimal extreme point solution for a given linear objective function under certain conditions.
\begin{proposition}[\cite{Nemhauser1999}, Part I.4 Theorem 4.5] \label{thm:extrp}
If $P \neq \emptyset$, $rank(D)=n$, and $\max \{c^Tv:v\in P\}$ is finite, then there is an optimal solution that is an extreme point.
\end{proposition}
The following proposition ensures that the subproblem finds extreme rays of the cone.
\begin{proposition}\label{thm:mpexr}
The subproblem \eqref{eq:SP} has at least one optimal extreme point. If the optimal value of \eqref{eq:SP} is negative, then an optimal extreme point of \eqref{eq:SP} corresponds to an extreme ray of the flux cone \eqref{eq:PolyP}.
\end{proposition}
\begin{proof}Because of the constraints $\mathbf{1^T}e \leq 1$ and $e \geq 0$ problem \eqref{eq:SP} is finite. Furthermore, the corresponding constraint matrix, given by \eqref{eq:PolyP1}, has full column rank because of $e \geq 0$. 
Hence by Proposition \ref{thm:extrp} there is an optimal solution to \eqref{eq:SP} that is an extreme point of \eqref{eq:PolyP1}.
For the remainder of the proof assume $e^*$ is an optimal extreme point with $c^Te^*<0$.
From Propositions 2.4, and 4.2 in \citet[Part I.4]{Nemhauser1999} an extreme point has $n$ linearly independent active constraints, \ie constraints satisfied with an equality. When the objective function value $c^Te^*$ is negative $\mathbf{1^T}e^* \leq 1$ is one of the active constraints in \eqref{eq:PolyP1}. 
This is since 
any $\tilde{e} \in$ \eqref{eq:PolyP1} such that $\mathbf{1^T}\tilde{e} < 1$ and with $c^T\tilde{e}<0$ can always be multiplied by an $\alpha >1$ so $\alpha \mathbf{1^T}\tilde{e} = 1$ and $\alpha c^T\tilde{e}<c^T\tilde{e}$.
Since $\mathbf{1^T}e \leq 1$ is the only constraint in \eqref{eq:PolyP1} that is not in \eqref{eq:PolyP} there are $n-1$ linearly independent constraints in \eqref{eq:PolyP} that are active for $e^*$. By Proposition 2.4 and 4.3 in \citep[Part I.4]{Nemhauser1999} a point in a cone is an extreme ray if and only if it has $n-1$ independent active constraints in the cone, hence $e^*$ is an extreme ray of \eqref{eq:PolyP}. 
\qed
\end{proof}
Hence, any method that generates extreme points of \eqref{eq:PolyP1}, \eg the simplex method, generates extreme rays of \eqref{eq:PolyP}. Then by Lemma \ref{lem:EREFM} the extreme rays of \eqref{eq:PolyP} are equal to the EFMs when the metabolic network only has irreversible reactions.
\subsection{Handling Reversible Reactions} \label{sc:irrev}
Thus far all metabolic networks were assumed to be irreversible. Networks with reversible reactions can be handled by the proposed method in Section \ref{sec:Method}, provided that a pre- and post-processing step is added. The pre-processing step involves moving to an extended space, where all reactions are irreversible, and the proposed method can be applied to find extreme rays of the cone in the extended space. The post-processing step maps these extreme rays back to the original space where they are EFMs. These steps have been described by \citet{Gagneur2004} and \citet{Urbanczik2005}.

The pre-processing step consists of splitting each reversible reaction into two irreversible reactions one in each direction. 
This can be considered as moving into an extended space, where the columns of $A_i$ and $A_x$ are split into reactions corresponding to reversible reactions ($A_i^{\text{rev}}$, $A_x^{\text{rev}}$) and irreversible reactions ($A_i^{\text{irrev}}$, $A_x^{\text{irrev}}$). Then an extended stoichiometric matrix $\tilde{A}$ is defined,
\begin{align*}
\tilde{A}_x=[A_x,\; \text{-}A_x^{\text{rev}}],\qquad
\tilde{A}_i=[A_i,\; \text{-}A_i^{\text{rev}}].
\end{align*}
In this extended space there is a correspondingly longer flux vector $\tilde{v}$ that consists of only irreversible reactions, or one directional edges in the network graph. 
The resulting cone,
\begin{equation*}
\{\, \tilde{v}: \; \tilde{A}_i \tilde{v} \leq 0, \; -\tilde{A}_i \tilde{v} \leq 0, \; -I \tilde{v} \leq 0 \,\}
\end{equation*}
has extreme rays. The master and subproblem combination can be applied to find extreme rays in this extended space. 

The post-processing step maps the extreme rays found in the extended space back by,
\begin{equation*}
\begin{aligned}
&v_j=\tilde{v}_j, \; \text{ if } j \in J_{\text{irrev}}, \quad \text{and} \\ &
v_j=\tilde{v}_j-\tilde{v}_{j'}, \; \text{ if } j \notin J_{\text{irrev}},
\end{aligned}
\end{equation*}
where $\tilde{v}_{j'}$ corresponds to the other direction of $\tilde{v}_j$. Then $v$ is either
\begin{enumerate}
\item equal to an EFM of the network, or 
\item an unimportant cycle in the extended network where the split reversible reaction is traversed in both forward and backwards direction \citep{Gagneur2004}.
\end{enumerate} 
In general the unimportant cycles will not be found using the master and subproblem combination, they do not give any improvement in the data fitting.
\section{Case-Study: Cultivation of CHO Cells} \label{sec:CS}
\subsection{Data}
\begin{table*}[t]
\centering
\scriptsize
\rowcolors{1}{white}{kthgrey}
\begin{tabular}{l|rrrrrr}
Metabolite & $q_{ext,6}$&$q_{ext,7}$&$q_{ext,8}$&$q_{ext,9}$&$q_{ext,10}$&$q_{ext,11}$ \\ \hline
Ala & 0.495 & 0.408 & 0.404 & 0.442 & 0.394 & 0.363\\
Arg & -0.307 & -0.251 & -0.121 & -0.241 & -0.187 & -0.309\\
Asn & -0.181 & -0.153 & -0.131 & -0.168 & -0.166 & -0.152\\
Asp & 0.053 & 0.036 & 0.054 & 0.062 & 0.065 & 0.043\\
Biomass & 0.633 & 0.536 & 0.602 & 0.669 & 0.612 & 0.615\\
Cys & -0.089 & -0.069 & -0.062 & -0.072 & -0.062 & -0.065\\
Glucose (Glc) & -4.454 & -3.114 & -3.037 & -3.721 & -3.591 & -2.624\\
Gln & -2.195 & -2.094 & -1.720 & -2.206 & -2.031 & -1.857\\
Glu & 0.285 & 0.298 & 0.257 & 0.286 & 0.310 & 0.272\\
Gly & 0.100 & 0.088 & 0.086 & 0.103 & 0.111 & 0.086\\
His & -0.062 & -0.064 & Na & Na & Na & Na\\
Ile & -0.120 & -0.107 & -0.084 & -0.110 & -0.119 & -0.091\\
Lactate (Lac) & 6.520 & 6.400 & 5.631 & 6.743 & 7.088 & 5.927\\
Leu & -0.204 & -0.179 & -0.145 & -0.195 & -0.204 & -0.150\\
Lys & -0.054 & -0.070 & -0.030 & -0.038 & -0.076 & -0.052\\
Met & -0.053 & -0.054 & -0.039 & -0.051 & -0.047 & -0.046\\
$\text{NH}_4^+$ & 1.298 & 1.251 & 1.115 & 1.255 & 1.236 & 1.143\\
Phe & -0.095 & -0.084 & -0.070 & -0.091 & -0.092 & -0.068\\
Pro & -0.108 & -0.119 & -0.086 & -0.122 & -0.133 & -0.098\\
Ser & -0.087 & -0.078 & -0.071 & -0.087 & -0.101 & -0.049\\
Thr & -0.119 & -0.120 & -0.085 & -0.110 & -0.093 & -0.130\\
Trp & -0.024 & -0.023 & -0.022 & -0.022 & -0.026 & -0.029\\
Tyr & -0.092 & -0.079 & -0.060 & -0.086 & -0.069 & -0.074\\
Val & -0.163 & -0.153 & -0.122 & -0.156 & -0.165 & -0.127\\
\end{tabular}
\caption{External fluxes based on the cell specific rates of consumption and production of a CHO cell line cultivated in pseudo-perfusion mode. 
The fluxes presented for each metabolite correspond to the final six days of the cultivation.
The unit of the fluxes is $\text{pmol}\cdot \text{cell}^{-1} \cdot \text{day}^{-1}$, except for Biomass for which the unit is $\text{day}^{-1}$.}
\label{tb:data}
\end{table*}
CHO cells were cultivated using pseudo-perfusion mode in spin tube bioreactors in a chemically defined medium over a time period of 11 days. The working volume was 10 mL. The medium was completely renewed on a daily basis, and the cells were re-seeded at $2 \cdot 10^6$ cells/mL to imitate steady-state conditions. The concentrations of $23$ external metabolites along with the cell number were analyzed before and after each medium exchange. Cell number, glucose, lactate, glutamine, glutamate, and $\text{NH}^+_4$ were analyzed using a BioProfile FLEX analyzer (Nova Biomedical, USA). Other amino acids were analyzed using Waters AccQ Tag Reagent Kit (Waters, USA) and high-performance liquid chromatography. Data from the final six days were used to calculate the cell specific growth rate along with the cell specific rates of consumption and production between each medium renewal. The resulting data set of external fluxes characterizing the metabolic state of the cells are presented in Table \ref{tb:data}.
\subsection{Metabolic Networks}
Two networks of varied complexity were used:
\begin{description}
\item[Network 1:] 36 reactions, (whereof 7 reversible), 31 metabolites (whereof 22 external), 
\item[Network 2:] 100 reactions, (whereof 29 reversible), 96 metabolites (whereof 24 external).
\end{description}
Network 1 was constructed based on information available in the online KEGG database \citep{kanehisa2000,kanehisa2012}, biochemistry literature \citep{Nelson2005}, and published metabolic networks of mammalian and CHO cell metabolism, see \eg 
\citet{Altamirano2001,Zamorano2010,Goudar2010,Ahn2011}.
The network consists of reactions from the glycolysis, the citric acid cycle, amino acid metabolism, and the synthesis of biomass from individual metabolites.
Unbalanced energy metabolites, which are involved in many reactions, were not considered. Furthermore, the metabolite $\text{CO}_2$ was not considered in this network representation. 

Network 2 was obtained from literature \citep[Section 2.2]{ZamoranoRiveros2012}. The network in its original format consists of 100 reactions and is one of the more detailed metabolic network representations of CHO cell metabolism available to date. 
The network consists of reactions from the glycolysis, the citric acid cycle, amino acid metabolism, the pentose phosphate pathway, the urea cycle, and lipid metabolism, as well as reactions representing the synthesis of nucleotides, fatty acids, proteins, and biomass. 
Reactions representing the transport of external metabolites between the internal and external space are also included.
The network was adapted for the present work by including transport reactions for additional external fluxes available in the data set. As for Network 1, unbalanced energy metabolites were not considered; the metabolite $\text{CO}_2$ was however kept. In the original network, 19 out of 100 reactions are defined as reversible. For the present work, 29 out of 100 reactions were defined as reversible, based on knowledge of mammalian cell metabolism.

Network 1 was constructed so that Metatool \citep{VonKamp2006} could enumerate the whole set of EFMs prior to solving the EFMs-based MFA problem. In contrast, the high complexity of Network 2 rendered Metatool unable to enumerate the whole set of EFMs. The metabolites arginine and histidine were absent in Network 1. Hence fewer measurements ($Q$) were available for solving the EFMs-based MFA problem compared to Network 2. 
\subsection{Results}
The master \eqref{eq:MP} and subproblem \eqref{eq:SP} combination was implemented and solved in {\scshape matlab}, version 2013a. The solver was {\scshape cplex} 10.2 ({\scshape ilog}, Sunnyvale, California), linked to {\scshape matlab} through an interface. The implementation was run on a computer with 64-bit Linux, a single Intel Xeon 3GHz processor core with hyper-threading disabled and 32GB of memory. 
It should be noted that the method was not implemented with run time efficiency as a priority.
\subsubsection{Results for Network 1}
The master and subproblem combination identified 14 EFMs needed to describe the data. The EFMs-based MFA problem was solved and the EFMs found in around 7 seconds. In comparison the total number of EFMs enumerated by Metatool was 257. The total run time of enumeration followed by solving the EFMs-based MFA problem was around 2 seconds. 
During the master and subproblem iterations, one supplementary EFM that was found but then disregarded since it did not contribute to the optimal solution. The macroscopic reactions of each EFM and the corresponding macroscopic fluxes are shown in Table \ref{tb:N1}. Both methods achieved the same optimal objective value of 2.54, \ie the 2-norm of the residual, $\norm{Q-\mathcal{I}A_xE_Bw_B}_2$. The optimal solution vectors were however different, \ie Metatool did not give the same weight $w$ on each macroscopic reaction as the solution found through the master and subproblem combination.
\begin{table}[htb]
\centering
\scriptsize
\rowcolors{1}{white}{kthgrey}
\begin{tabular}{c|p{0.65\linewidth}|r}
\bf{EFM:} & \bf{Macroscopic Reaction} & $w$ \\ \hline
1 & 0.5 Glc + 1 $\text{NH}_4^+$ $\Rightarrow$ 1 Ser & 6.82\\
2 & 1 Ser $\Rightarrow$ 1 Lac + 1 $\text{NH}_4^+$ & 5.99\\
3 & 1 Asp + 1 Gln $\Leftrightarrow$ 1 Asn + 1 Glu & 1.67\\
4 & 1 Asn $\Leftrightarrow$ 1 Asp + 1 $\text{NH}_4^+$ & 1.50\\
5 & 1 Glu $\Rightarrow$ 1 Ala & 0.83\\
6 & 1 Ser $\Leftrightarrow$ 1 Gly & 0.69\\
7 & 0.0208 Glc + 0.026 Asp + 0.0004 Cys + 0.0165 Gly + 0.0084 Ile + 0.0133 Leu + 0.0101 Lys + 0.0033 Met + 0.0055 Phe + 0.0081 Pro + 0.0099 Ser + 0.008 Thr + 0.004 Trp + 0.0077 Tyr + 0.0096 Val + 0.0006 Glu + 0.0133 Ala + 0.0377 Gln $\Rightarrow$ 1 Biomass & 0.64\\
8 & 1 Gly $\Rightarrow$ 1 $\text{NH}_4^+$ & 0.42\\
9 & 1 Glu $\Rightarrow$ 1 Asp & 0.41\\
10 & 1 Ala $\Rightarrow$ 1 Lac + 1 $\text{NH}_4^+$ & 0.24\\
11 & 1 Cys $\Rightarrow$ 1 Lac & 0.12\\
12 & 1 Met + 1 Ser $\Rightarrow$ 1 Lac + 1 Cys + 1 $\text{NH}_4^+$ & 0.04\\
13 & 2/3 Asp + 1/3 Leu + 1/3 Met + 1/3 Ser $\Rightarrow$ 1/3 Cys + 1 Glu + 1/3 $\text{NH}_4^+$ & 0.02\\
14 & 1 Phe $\Rightarrow$ 1 Tyr & 0.004\\
\end{tabular}
\caption{Results for Network 1, macroscopic reactions corresponding to each EFM needed for EFMs-based MFA, ordered by size of the macroscopic flux ($w$) in the unit $\text{pmol}\cdot \text{cell}^{-1} \cdot \text{day}^{-1}$}
\label{tb:N1}
\end{table}
\subsubsection{Results for Network 2}
The master and subproblem combination identified 25 EFMs needed to describe the data. The EFMs-based MFA problem was solved and the EFMs found in around 7 seconds. All EFMs found during the iterations contributed to the optimal solution. 
Metatool spent around 1 second trying to enumerate all EFMs but then returned an error message indicating that the network was too complex to be handled. The macroscopic reactions and the corresponding macroscopic fluxes are shown in Table \ref{tb:N3}. The solution gave an optimal objective value of 1.9, \ie the  2-norm of the residual, $\norm{Q-\mathcal{I}A_xE_Bw_B}_2$.
\begin{table*}[hbt]
\scriptsize
\centering
\rowcolors{1}{white}{kthgrey}
\begin{tabular}{c|p{0.75\linewidth}|r}
\bf{EFM:} & \bf{Macroscopic Reaction} & $w$ \\ \hline
1 & 0.5 Glu $\Rightarrow$ 0.5 Ala + 1 CO2 & 6.97\\
2 & 0.5 Glc $\Rightarrow$ 1 Lac & 4.86\\
3 & 1 Ala + 1 CO2 $\Rightarrow$ 1 Asp & 2.64\\
4 & 0.5 Glc + 0.5 Asn + 0.5 Ala $\Rightarrow$ 1 Ser + 0.5 Glu + 1 CO2 & 1.95\\
5 & 1 Asn $\Rightarrow$ 1 Lac & 1.65\\
6 & 1 Ser $\Leftrightarrow$ 1 Gly & 1.56\\
7 & 1 Gln + 1 Asp $\Rightarrow$ 1 Asn + 1 Glu & 1.41\\
8 & 1 Asp + 1 Gly $\Rightarrow$ 1 Asn + 1 CO2 & 0.94\\
9 & 1 Gln $\Rightarrow$ 1 Glu + 1 $\text{NH}_4^+$ & 0.56\\
10 & 1 Lac + 1 Gly $\Rightarrow$ 1 Ala + 1 CO2 & 0.51\\
11 & 0.028217 Glc + 0.0656 Gln + 0.0686 Ser + 0.0468 Asn + 0.046 Asp + 0.046 Arg + 0.0276 Tyr + 0.1296 Thr + 0.0552 Lys + 0.0644 Val + 0.046 Ile + 0.2348 Leu + 0.0368 Phe + 0.0184 Met + 0.1364 Lac + 0.0264 Glu + 0.0184 Cys + 0.0184 His + 0.046 Pro + 0.0092 Trp + 0.006 Ethanolamine + 0.0171 Choline $\Rightarrow$ 0.0628 Ala + 0.0212 CO2 + 1 Biomass & 0.45\\
12 & 1 Ser $\Rightarrow$ 1 Lac + 1 $\text{NH}_4^+$ & 0.39\\
13 & 1 Arg $\Rightarrow$ 1 Glu + 1 $\text{NH}_4^+$ & 0.21\\
14 & 2/3 Asp + 1/3 Leu + 1/3 Lac $\Rightarrow$ 1 Glu + 2/3 CO2 & 0.18\\
15 & 2/3 Val $\Rightarrow$ 1/3 Asn + 1/3 Lac + 1 CO2 & 0.16\\
16 & 1/3 Thr + 1/3 Phe $\Rightarrow$ 2/3 Glu + 1 CO2 & 0.13\\
17 & 1/3 Ser + 1/3 Tyr + 1/3 Met $\Rightarrow$ 2/3 Glu + 1 CO2 + 1/3 Cys & 0.11\\
18 & 1 Asp + 1 Cys $\Rightarrow$ 1 Asn + 1 Lac & 0.10\\
19 & 1 Pro $\Rightarrow$ 1 Glu & 0.08\\
20 & 0.025329 Glc + 0.058887 Gln + 0.06158 Ser + 0.033034 Asn + 0.064273 Asp + 0.041293 Arg + 0.049551 Thr + 0.065889 Lys + 0.05781 Val + 0.041293 Ile + 0.074327 Leu + 0.05781 Phe + 0.016517 Met + 0.0096948 Lac + 0.066786 Gly + 0.016517 Cys + 0.016517 His + 0.041293 Pro + 0.24183 Trp + 0.005386 Ethanolamine + 0.01535 Choline $\Rightarrow$ 0.1675 Ala + 1 CO2 + 0.89767 Biomass & 0.08\\
21 & 1 Ile $\Rightarrow$ 1 Glu + 1 CO2 & 0.08\\
22 & 0.024997 Glc + 0.058115 Gln + 0.060773 Ser + 0.05528 Asn + 0.040751 Asp + 0.040751 Arg + 0.50815 Tyr + 0.048901 Thr + 0.048901 Lys + 0.057052 Val + 0.040751 Ile + 0.073352 Leu + 0.032601 Phe + 0.0163 Met + 0.065202 Ala + 0.065911 Gly + 0.0163 Cys + 0.0163 His + 0.040751 Pro + 0.0081502 Trp + 0.0053154 Ethanolamine + 0.015149 Choline $\Rightarrow$ 0.47413 Glu + 1 CO2 + 0.8859 Biomass & 0.05\\
23 & 1 His $\Rightarrow$ 1 Glu + 1 $\text{NH}_4^+$ & 0.05\\
24 & 0.028217 Glc + 0.0656 Gln + 0.0686 Ser + 0.0368 Asn + 0.046 Arg + 0.0552 Thr + 0.0552 Lys + 0.0644 Val + 0.046 Ile + 0.0828 Leu + 0.4396 Phe + 0.0184 Met + 0.0744 Gly + 0.0184 Cys + 0.0184 His + 0.046 Pro + 0.0092 Trp + 0.006 Ethanolamine + 0.0171 Choline $\Rightarrow$ 0.0256 Lac + 0.1936 Glu + 0.7156 CO2 + 1 Biomass & 0.04\\
25 & 0.5 Lys + 1 Lac $\Rightarrow$ 1 Glu + 1 CO2 & 0.04\\
\end{tabular}
\caption{Results for Network 2, macroscopic reactions corresponding to each EFM needed for EFMs-based MFA, ordered by size of the macroscopic flux ($w$) in the unit $\text{pmol}\cdot \text{cell}^{-1} \cdot \text{day}^{-1}$}
\label{tb:N3}
\end{table*}
\subsubsection{Data Fitting Performance}
The fit for each external metabolite was quantified and is presented in Figure \ref{fig:ResExt} in the following way. 
For each metabolite the 2-norm of the residual of all measurement was examined. 
Specifically let $Q_{mb}$ be the vector of measurements and $A_{x,mb}$ the rows of the stoichiometric matrix corresponding to a specific metabolite. Then Figure \ref{fig:ResExt} shows $\norm{Q_{mb}-\mathcal{I}_{mb}A_{x,mb}E_Bw_B}_2$ for each metabolite and network.
\begin{figure}[htb]
\centering
\includegraphics[width=\linewidth]{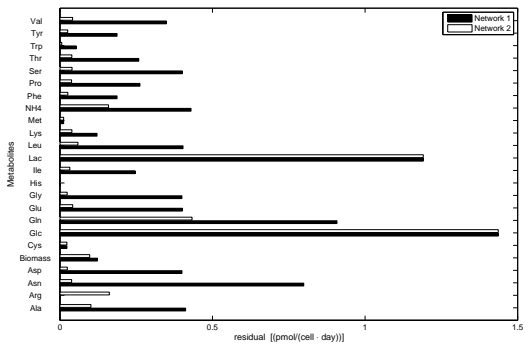}
\caption{The 2-norm of the residual for each external metabolite in Network 1 and 2}
\label{fig:ResExt}
\end{figure}
The comparison in Figure \ref{fig:ResExt} shows that several external metabolites were fitted better by Network 2, while a few remain at the same level as with Network 1. 
\section{Conclusion} \label{sec:conc}
We have presented a method that solves the EFMs-based MFA problem without enumerating the whole set of EFMs.
The method was derived using column generation. This technique is well established in other fields; however, its application to EFMs-based MFA is to the best of our knowledge new. 
Through the use of this technique, a smaller set of EFMs is generated while still ensuring convergence of the EFMs-based MFA problem to an optimal solution, even for a complex network. In other words, the method returns the same optimal objective value that would have been achieved if the whole set of EFMs was known. This enables more complex networks to be used when analyzing the cell metabolism. 

It should be noted that the subset of EFMs obtained through our method is not necessarily unique. Hence, another subset of EFMs could be found that fits the data in an equally good way. Nevertheless, obtaining a small set of EFMs relevant to external measurements based on a complex network can provide an indication of what pathways are important. 
To obtain a unique and globally optimal solution 
a regularized version of the EFMs-based MFA problem could be solved, \ie $ \frac{1}{2} \norm{Q-\mathcal{I}A_{x}Ew}_2 +\frac{1}{2} \xi\norm{w}_2 $ where $\xi$ is some small enough positive number. However, the regularization of $w$ with respect to the 2-norm will predispose the solution to include 
a higher number of EFMs, each with lower flux ($w_l$). 
In modeling applications, a smaller set of relevant EFMs may be more desirable since it could mean that fewer parameters need to be estimated.
Furthermore, the unique solution obtained through the regularization does not necessarily imply that solution to be more relevant than other solutions of the EFMs-based MFA problem.

There exist diverse tools to find EFMs. We chose to compare the proposed column generation method to Metatool, an EFM enumeration tool that has previously been used to solve the EFMs based MFA problem \citep{Provost2006}. The larger network was thus selected as complex enough so that the comparison tool would fail. While other EFM enumeration tool exist, \eg \citet{Terzer2008}, any such tool will become infeasible with increased network complexity. 
Consequently, focus has shifted from total enumeration of EFMs for complex networks to extracting a subset with certain desired properties. 
The LP search algorithm methods presented by \citet{Kaleta2009} and later by \citet{Tabe-Bordbar2013}, are appropriate when the goal is to, \eg identify as many EFMs as possible. However, the subset extracted by the LP search algorithm methods is not guaranteed to give a globally optimal solution to the EFMs-based MFA problem.
In contrast the method presented by \citet{Jungers2011} was specifically developed for solving the EFMs-based MFA problem. It generates a subset of EFMs through a series of optimization problems using a pre-calculated flux vector $v$ as a guide. 
Compared to \citet{Jungers2011} the column generation method does not require the pre-calculation of the flux vector $v$ since external measurements are taken directly into account. 
Additionally, we can find solutions with fewer EFMs than the predetermined minimal number found by \citet{Jungers2011}.
Finally, compared to the previously mentioned methods our method presented a more integrated approach, where identifying EFMs is done in conjunction with solving the EFMs-based MFA problem.

In the case-study, small subsets of EFMs and corresponding macroscopic fluxes were obtained for both networks at a low computational cost. 
The case study implies that the method is effective, since there was only a single unnecessary iteration for Network 1 and no unnecessary iteration for Network 2. This indicates that the EFMs-based MFA problem, with Network 2, was solved to optimality through solving 26 LP's and 25 quadratic optimization problems.
Even though the networks have many EFMs only a fraction of those were needed to fit the data in an optimal way. 
Our method allowed the use of a network of higher complexity, which improved the fit between estimated fluxes and data as the norm for several external fluxes was greatly reduced. This result highlights the value of being able to keep a certain level of network complexity while avoiding simplifications forced by computational limitations. 

Our approach has the potential to find relevant EFMs of even larger networks than is presented in this work, thereby increasing the size of problems that can be handled by EFMs-based MFA.
Furthermore, we believe that the strategy of generating EFMs using column generation has many possible applications. 
One example is the identification of a subset of EFMs and parameters of a dynamic model that best fits experimental data.
Another example could be to develop strategies that take measurement errors more specifically into account than through least-squares approximation.
\renewcommand{\abstractname}{Acknowledgements}
\begin{abstract}
The work of the authors from the Department of Mathematics was supported by the Swedish Research Council. The work of the authors from the Division of Industrial Biotechnology was supported by KTH and the Swedish Governmental Agency for Innovation Systems (VINNOVA). The CHO cell line was kindly provided by Selexis (Switzerland). Culture media were kindly provided by Irvine Scientific (CA, USA). 
Finally, we thank the editor and the two anonymous referees for their valuable comments and suggestions. 
\end{abstract}

\bibliographystyle{natbib}

\begin{thebibliography}{37}
\providecommand{\natexlab}[1]{#1}
\providecommand{\url}[1]{{#1}}
\providecommand{\urlprefix}{URL }
\expandafter\ifx\csname urlstyle\endcsname\relax
  \providecommand{\doi}[1]{DOI~\discretionary{}{}{}#1}\else
  \providecommand{\doi}{DOI~\discretionary{}{}{}\begingroup
  \urlstyle{rm}\Url}\fi
\providecommand{\eprint}[2][]{\url{#2}}

\bibitem[{Acu\~{n}a et~al(2009)Acu\~{n}a, Chierichetti, Lacroix,
  Marchetti-Spaccamela, Sagot, and Stougie}]{Acuna2009}
Acu\~{n}a V, Chierichetti F, Lacroix V, Marchetti-Spaccamela A, Sagot MF,
  Stougie L (2009) {Modes and cuts in metabolic networks: complexity and
  algorithms.} Biosystems 95(1):51--60

\bibitem[{Ahn and Antoniewicz(2011)}]{Ahn2011}
Ahn WS, Antoniewicz MR (2011) Metabolic flux analysis of \{CHO\} cells at
  growth and non-growth phases using isotopic tracers and mass spectrometry.
  Metabolic Engineering 13(5):598 -- 609

\bibitem[{Altamirano et~al(2001)Altamirano, Illanes, Casablancas, Gámez,
  Cairó, and Gòdia}]{Altamirano2001}
Altamirano C, Illanes A, Casablancas A, Gámez X, Cairó JJ, Gòdia C (2001)
  Analysis of cho cells metabolic redistribution in a glutamate-based defined
  medium in continuous culture. Biotechnology Progress 17(6):1032--1041

\bibitem[{Bonarius and Schmid(1997)}]{Bonarius1997}
Bonarius HPJ, Schmid G (1997) {Flux analysis of underdetermined metabolic
  networks : the quest for the missing constraints}. Trends in Biotechnology
  15(8):308--314

\bibitem[{Clarke(1980)}]{Clarke1980}
Clarke BL (1980) Stability of Complex Reaction Networks, vol~43. Advances in
  Chemical Physics

\bibitem[{de~Figueiredo et~al(2009)de~Figueiredo, Podhorski, Rubio, Kaleta,
  Beasley, Schuster, and Planes}]{DeFigueiredo2009}
de~Figueiredo LF, Podhorski A, Rubio A, Kaleta C, Beasley JE, Schuster S,
  Planes FJ (2009) {Computing the shortest elementary flux modes in
  genome-scale metabolic networks.} Bioinformatics (Oxford, England)
  25(23):3158--65

\bibitem[{Gagneur and Klamt(2004)}]{Gagneur2004}
Gagneur J, Klamt S (2004) {Computation of elementary modes: a unifying
  framework and the new binary approach.} BMC bioinformatics 5:175

\bibitem[{Goudar et~al(2010)Goudar, Biener, Boisart, Heidemann, Piret,
  de~Graaf, and Konstantinov}]{Goudar2010}
Goudar C, Biener R, Boisart C, Heidemann R, Piret J, de~Graaf A, Konstantinov K
  (2010) Metabolic flux analysis of \{CHO\} cells in perfusion culture by
  metabolite balancing and 2d [13c, 1h] \{COSY\} \{NMR\} spectroscopy.
  Metabolic Engineering 12(2):138 -- 149

\bibitem[{Griva et~al(2009)Griva, Nash, and Sofer}]{Griva2009}
Griva I, Nash SG, Sofer A (2009) Linear and Nonlinear Optimization, 2nd edn.
  Society for Industrial Mathematics

\bibitem[{Jungers et~al(2011)Jungers, Zamorano, Blondel, {Vande Wouwer}, and
  Bastin}]{Jungers2011}
Jungers RM, Zamorano F, Blondel VD, {Vande Wouwer} A, Bastin G (2011) {Fast
  computation of minimal elementary decompositions of metabolic flux vectors}.
  Automatica 47(6):1255--1259

\bibitem[{Kaleta et~al(2009)Kaleta, de~Figueiredo, Behre, and
  Schuster}]{Kaleta2009}
Kaleta C, de~Figueiredo L, Behre J, Schuster (2009) {EFMEvolver: Computing
  elementary flux modes in genome-scale metabolic networks}. In: Grosse I,
  Neumann S, Posch S, Schreiber F, Stadler P (eds) Lecture Notes in Informatics
  P-157, Gesellschaft f\"{u}r Informatik, Bonn, pp 179--189

\bibitem[{von Kamp and Schuster(2006)}]{VonKamp2006}
von Kamp A, Schuster S (2006) {Metatool 5.0: fast and flexible elementary modes
  analysis.} Bioinformatics (Oxford, England) 22(15):1930--1

\bibitem[{Kanehisa and Goto(2000)}]{kanehisa2000}
Kanehisa M, Goto S (2000) {KEGG: Kyoto Encyclopedia of Genes and Genomes}.
  Nucleic Acids Research 28(1):27--30

\bibitem[{Kanehisa et~al(2012)Kanehisa, Goto, Sato, Furumichi, and
  Tanabe}]{kanehisa2012}
Kanehisa M, Goto S, Sato Y, Furumichi M, Tanabe M (2012) {KEGG for integration
  and interpretation of large-scale molecular data sets.} Nucleic acids
  research 40(Database issue):D109--D114

\bibitem[{Klamt and Schuster(2002)}]{Klamt2002}
Klamt S, Schuster S (2002) {Calculability analysis in underdetermined metabolic
  networks illustrated by a model of the central metabolism in purple nonsulfur
  bacteria}. Biotechnology and Bioengineering 77(7):734--751

\bibitem[{Klamt and Stelling(2002)}]{Klamt2002a}
Klamt S, Stelling J (2002) {Combinatorial complexity of pathway analysis in
  metabolic networks.} Molecular biology reports 29(1-2):233--6

\bibitem[{Klamt and Stelling(2003)}]{Klamt2003}
Klamt S, Stelling J (2003) {Two approaches for metabolic pathway analysis?}
  Trends in biotechnology 21(2):64--9

\bibitem[{Llaneras and Pic\'{o}(2008)}]{Llaneras2008}
Llaneras F, Pic\'{o} J (2008) {Stoichiometric modelling of cell metabolism.} J
  Biosci Bioeng 105(1):1--11

\bibitem[{Llaneras and Pic\'{o}(2010)}]{Llaneras2010}
Llaneras F, Pic\'{o} J (2010) {Which metabolic pathways generate and
  characterize the flux space? A comparison among elementary modes, extreme
  pathways and minimal generators.} Journal of biomedicine \& biotechnology
  2010:753,904

\bibitem[{L\"{u}bbecke and Desrosiers(2005)}]{Lubbecke2005}
L\"{u}bbecke ME, Desrosiers J (2005) {Selected Topics in Column Generation}.
  Operations Research 53(6):1007--1023

\bibitem[{Nelson and Cox(2004)}]{Nelson2005}
Nelson DL, Cox MM (2004) {Lehninger Principles of Biochemistry, Fourth
  Edition}, 4th edn. W. H. Freeman

\bibitem[{Nemhauser and Wolsey(1999)}]{Nemhauser1999}
Nemhauser GL, Wolsey LA (1999) Integer and Combinatorial Optimization, 1st edn.
  John Wiley \& Sons, inc

\bibitem[{Papin et~al(2003)Papin, Price, Wiback, Fell, and Palsson}]{Papin2003}
Papin Ja, Price ND, Wiback SJ, Fell Da, Palsson BO (2003) {Metabolic pathways
  in the post-genome era.} Trends in biochemical sciences 28(5):250--8

\bibitem[{Papin et~al(2004)Papin, Stelling, Price, Klamt, Schuster, and
  Palsson}]{Papin2004}
Papin Ja, Stelling J, Price ND, Klamt S, Schuster S, Palsson BO (2004)
  {Comparison of network-based pathway analysis methods.} Trends in
  biotechnology 22(8):400--5

\bibitem[{Planes and Beasley(2008)}]{Planes2008}
Planes FJ, Beasley JE (2008) {A critical examination of stoichiometric and
  path-finding approaches to metabolic pathways.} Briefings in bioinformatics
  9(5):422--36

\bibitem[{Price et~al(2003)Price, Reed, Papin, Famili, and
  Palsson}]{Price2003a}
Price ND, Reed JL, Papin JA, Famili I, Palsson BO (2003) {Analysis of metabolic
  capabilities using singular value decomposition of extreme pathway matrices.}
  Biophysical journal 84(2 Pt 1):794--804

\bibitem[{Provost(2006)}]{Provost2006}
Provost A (2006) {Metabolic Design of Dynamic Bioreaction Models}. PhD thesis,
  Universit\'{e} catholique de Louvain

\bibitem[{Provost et~al(2006)Provost, Bastin, Agathos, and
  Schneider}]{Provost2006a}
Provost A, Bastin G, Agathos SN, Schneider YJ (2006) {Metabolic design of
  macroscopic bioreaction models: application to Chinese hamster ovary cells.}
  Bioprocess and biosystems engineering 29(5-6):349--66

\bibitem[{Rezola et~al(2011)Rezola, de~Figueiredo, Brock, Pey, Podhorski,
  Wittmann, Schuster, Bockmayr, and Planes}]{Rezola2011}
Rezola A, de~Figueiredo LF, Brock M, Pey J, Podhorski A, Wittmann C, Schuster
  S, Bockmayr a, Planes FJ (2011) {Exploring metabolic pathways in genome-scale
  networks via generating flux modes.} Bioinformatics 27(4):534--40

\bibitem[{Schilling et~al(1999)Schilling, Schuster, Palsson, and
  Heinrich}]{Schilling1999}
Schilling CH, Schuster S, Palsson BO, Heinrich R (1999) {Metabolic pathway
  analysis: basic concepts and scientific applications in the post-genomic
  era.} Biotechnology progress 15(3):296--303

\bibitem[{Schilling et~al(2000)Schilling, Letscher, and
  Palsson}]{Schilling2000}
Schilling CH, Letscher D, Palsson BO (2000) {Theory for the systemic definition
  of metabolic pathways and their use in interpreting metabolic function from a
  pathway-oriented perspective.} Journal of theoretical biology 203(3):229--48

\bibitem[{Schuster and Hilgetag(1994)}]{Schuster1994}
Schuster S, Hilgetag C (1994) {On Elementary Flux Modes in Biochemical Reaction
  Systems at Steady State}. Journal of Biological Systems 2(2):165--182

\bibitem[{Tabe-Bordbar and Marashi(2013)}]{Tabe-Bordbar2013}
Tabe-Bordbar S, Marashi SA (2013) {Finding elementary flux modes in metabolic
  networks based on flux balance analysis and flux coupling analysis:
  application to the analysis of Escherichia coli metabolism.} Biotechnol Lett
  35(12):2039--44

\bibitem[{Terzer and Stelling(2008)}]{Terzer2008}
Terzer M, Stelling J (2008) {Large-scale computation of elementary flux modes
  with bit pattern trees.} Bioinformatics 24(19):2229--35

\bibitem[{Urbanczik and Wagner(2005)}]{Urbanczik2005}
Urbanczik R, Wagner C (2005) {An improved algorithm for stoichiometric network
  analysis: theory and applications.} Bioinformatics (Oxford, England)
  21(7):1203--10

\bibitem[{Zamorano et~al(2010)Zamorano, {Vande Wouwer}, and
  Bastin}]{Zamorano2010}
Zamorano F, {Vande Wouwer} A, Bastin G (2010) {A detailed metabolic flux
  analysis of an underdetermined network of CHO cells.} Journal of
  biotechnology 150(4):497--508

\bibitem[{{Zamorano Riveros}(2012)}]{ZamoranoRiveros2012}
{Zamorano Riveros} F (2012) {Metabolic Flux Analysis of CHO cell cultures}. PhD
  thesis, University of Mons

\end{thebibliography}

\end{document}